\documentclass[11pt,english]{smfart}
\usepackage[latin1]{inputenc}
\usepackage[T1]{fontenc}
\usepackage[english,french]{}
\usepackage{hyperref}
\usepackage{fancyhdr}
\usepackage{layouts}
\usepackage{layout}
\usepackage{url}
\usepackage{hyperref}

\let\sst=\scriptscriptstyle
\def\lien{\mathrel{\mkern-4mu}}
\def\too{\relbar\lien\rightarrow}
\def\tooo{\relbar\lien\relbar\lien\too}

\newtheorem{theorem}{Theorem}[section]
\newtheorem{lemma}{Lemma}[section]
\newtheorem{proposition}{Proposition}[section]
\newtheorem{corollary}{Corollary}[section]
\newtheorem{definition}{Definition}[section]
\newtheorem{remark}{Remark}[section]

\numberwithin{equation}{section}
\def\N{\mathbb{N}}

\def\Q{\mathbb{Q}}
\def\Z{\mathbb{Z}}
\def\F{\mathbb{F}}
\def\Frac#1#2{\hbox{\footnotesize $\displaystyle \frac{#1}{#2}$}}
\def\plus{\displaystyle\mathop{\raise 1.0pt \hbox{$\bigoplus $}}\limits}
\def\prd{\displaystyle\mathop{\raise 2.0pt \hbox{$\prod$}}\limits}
\def\sm{\displaystyle\mathop{\raise 2.0pt \hbox{$\sum$}}\limits}
\let\ds=\displaystyle
\let\wt=\widetilde
\let\ov=\overline
\def\Cl{{\mathcal C}\hskip-2pt{\ell}}

\def\order{\raise1.5pt \hbox{${\scriptscriptstyle \#}$}}

\begin{document}

\markboth{Georges Gras}
{The $p$-adic Kummer--Leopoldt Constant}

\title{The $p$-adic Kummer--Leopoldt Constant  \\  Normalized $p$-adic Regulator}

\author{Georges Gras}

\address{Villa la Gardette, Chemin Ch\^ateau Gagni\`ere
\\ F--38520 Le Bourg d'Oisans, France -- \url{https://www.researchgate.net/profile/Georges_Gras}}
\email{g.mn.gras@wanadoo.fr}

\maketitle

\begin{abstract}
The $p$-adic Kummer--Leopoldt constant $\kappa_{\sst K}$ of a number 
field $K$ is (assuming the Leopoldt conjecture) the least integer $c$ 
such that for all $n \gg 0$, any global unit 
of $K$, which is locally a $p^{n+c}$th power at the $p$-places, 
is necessarily the $p^n$th power of a global unit of~$K$. This constant has been 
computed by Assim \& Nguyen Quang Do using Iwasawa's techniques,
after intricate studies and calculations by many authors. 
We give an elementary $p$-adic proof and an improvement of these 
results, then a class field theory interpretation of $\kappa_{\sst K}$. 
We give some applications (including generalizations of Kummer's 
lemma on regular $p$th cyclotomic fields) and a natural definition of the 
normalized $p$-adic regulator for any $K$ and any $p\geq 2$.
This is done without analytical computations, using only class field theory
and especially the properties of the so-called $p$-torsion 
group ${\mathcal T}_K$ of Abelian $p$-ramification theory over $K$.
\end{abstract}

\begin{altabstract}
La constante $p$-adique de Kummer--Leopoldt $\kappa_{\sst K}$ d'un corps de nombres
$K$ est (sous la conjecture de Leopoldt) le plus petit entier $c$ tel que pour tout $n \gg 0$,
toute unit\'e globale de $K$, qui est localement une puissance $p^{n+c}$-i\`eme en les
$p$-places, est n\'ecessairement puissance $p^n$-i\`eme d'une unit\'e globale de $K$.
Cette constante a \'et\'e calcul\'ee par Assim \& Nguyen Quang Do en utilisant les techniques 
d'Iwasawa, apr\`es des \'etudes et calculs complexes par divers auteurs.
Nous donnons une preuve $p$-adique \'el\'ementaire et une g\'en\'eralisation 
de ces r\'esultats, puis une interpr\'etation corps de classes de $\kappa_{\sst K}$.
Nous donnons certaines applications (dont des g\'en\'eralisations du lemme de Kummer
sur les $p$-corps cyclotomiques r\'eguliers) et une d\'efinition naturelle du r\'egulateur
$p$-adique normalis\'e pour tous $K\  \& \  p \geq 2$. 
Ceci est fait sans calculs analytiques, en utilisant uniquement le corps de 
classes et tout sp\'ecialement les propri\'et\'es du fameux $p$-groupe de torsion 
${\mathcal T}_K$ de la th\'eorie de la $p$-ramification Ab\'elienne sur $K$.
\end{altabstract}

\keywords{Global units; $p$-adic regulator; Leopoldt conjecture; class field theory;
Abelian $p$-ramification; cyclotomic fields.}

\subjclass{Mathematics Subject Classification 2010: 11R27, 11R37, 11R29}

\section{Notations}\label{section1} Let $K$ be a number field and let
$p\geq 2$ be a prime number; we denote by ${\mathfrak p} \mid p$ 
the prime ideals of $K$ dividing $p$. 
Consider the group $E_K$ of $p$-principal global 
units of $K$ (i.e., units $\varepsilon \equiv 1 \! \pmod{ \prod_{{\mathfrak p} \mid p} {\mathfrak p}}$),
so that the index of $E_K$ in the group of units is prime to $p$.
For each ${\mathfrak p} \mid p$, let $K_{\mathfrak p}$ be the ${\mathfrak p}$-completion
of $K$ and $\ov {\mathfrak p}$ the corresponding prime ideal of the ring of integers 
of $K_{\mathfrak p}$; then let 

\smallskip
\centerline{$U_K := \Big \{u \in \plus_{{\mathfrak p} \mid p}K_{\mathfrak p}^\times, \ \,
u = 1+x, \  x \in \plus_{{\mathfrak p} \mid p} \ov {\mathfrak p} \Big\}\, \ \ \& \ \ \,
W_K := {\rm tor}_{\Z_p}^{}(U_K)$,} 

\smallskip\noindent
the $\Z_p$-module of principal local units at $p$ and its  torsion subgroup.

\smallskip
The $p$-adic logarithm ${\rm log}$ is defined on $1+x$, 
$x \in \plus_{{\mathfrak p} \mid p} \ov {\mathfrak p}$, 
by means of the series ${\rm log} (1+ x) = 
\sm_{i \ge 1}\, (-1)^{i+1} \, \Frac{x^i}{i} \in \plus_{{\mathfrak p} \mid p}K_{\mathfrak p}$.
Its kernel in $U_K$ is $W_K$ \cite[Proposition 5.6]{Wa1}.

\smallskip
We consider the diagonal embedding $E_K \tooo U_K$ and its natural extension 
$E_K \otimes \Z_p \tooo U_K$ whose image is $\ov E_K$, the topological 
closure of $E_K$ in $U_K$. 

\smallskip
In the sequel, these embeddings shall be understood;
moreover, we assume in this paper that $K$ satisfies the Leopoldt conjecture at $p$, 
which is equivalent to the condition ${\rm rk}_{\Z_p}(\ov E_K) = {\rm rk}_{\Z}(E_K)$
(see, e.g., \cite[\S\,5.5, p.\,75]{Wa1}).

\section{The Kummer--Leopoldt constant}\label{section2}

This notion comes from the Kummer lemma (see, e.g., \cite[Theorem 5.36]{Wa1}), 
that is to say, if the odd prime number $p$ is ``regular'', the cyclotomic field $K=\Q(\mu_p)$ 
of $p$th roots of unity satisfies the following property stated for the whole 
group $E'_K$ of global units of $K$: 

\smallskip
\centerline{\it any $\varepsilon \in E'_K$, congruent to a rational
integer modulo $p$, is a $p$th power in $E'_K$.}

\smallskip
In fact, $\varepsilon \equiv a \pmod p$ with $a \in \Z$, implies 
$\varepsilon^{p-1} \equiv a^{p-1} \equiv 1 \pmod p$. So we shall write the 
Kummer property with $p$-principal units in the more suitable equivalent statement:

\smallskip
\centerline{\it any $\varepsilon \in E_K$, congruent to $1$ modulo $p$, is a $p$th power in $E_K$.}

\medskip
From \cite{A}, \cite{L}, \cite{O}, \cite{S}, \cite{Wa2}, \cite{Wa3} one can study this property 
and its generalizations with various techniques (see the rather intricate history in \cite{AN}). 
Give the following definition from \cite{AN}:

\begin{definition} \label{defkappa}
Let $K$ be a number field satisfying the Leopoldt conjecture at the prime
$p\geq 2$. Let $E_K$ be the group of $p$-principal global units of $K$ and let
$U_K$ be the group of principal local units at the $p$-places.

\noindent
We call Kummer--Leopoldt constant (denoted $\kappa_{\sst K} =: \kappa$), 
the smallest integer $c$ such that the following condition is fulfilled:

\smallskip
\centerline{\it for all $n \gg 0$, any unit $\varepsilon \in E_K$, such that 
$\varepsilon \in U_K^{p^{n+c}}$, is necessarily in $E_K^{p^n}$.}
\end{definition}

\begin{remark}\label{remas}
The {\it existence} of $\kappa$ comes from various classical characterizations 
of Leopoldt's conjecture proved for instance in \cite[Theorem III.3.6.2]{Gr1},
after \cite{S}, \cite{L} and oldest Iwasawa papers.
Indeed, if the Leopoldt conjecture is not satisfied, we can find a sequence
$\varepsilon_n \in E_K \setminus E_K^p$ such that ${\rm log} (\varepsilon_n)
\to 0$ (i.e., $\varepsilon_n \in U_K^{p^m} \!\cdot W_K$, with $m \to \infty$ 
as $n\to\infty$); since $W_K$ is finite, taking a suitable $p$-power of 
$\varepsilon_n$, we see that $\kappa$ does not exist in that case.
\end{remark}

 We shall prove (Theorem \ref{nu}) that in the above definition, the condition
``for all $n \gg 0$'' can be replaced by ``for all $n \geq 0$'', subject to introduce
the group of global roots of unity of $K$ and a suitable statement.

\smallskip
We have the following first $p$-adic result giving $p^\kappa$ under 
the Leopoldt conjecture:

\begin{theorem} \label{thm} 
Denote by $E_K$ the group of $p$-principal global units of $K$, by $U_K$ 
the $\Z_p$-module of principal local units at the $p$-places, and by $W_K$ 
its torsion subgroup.
Let $\kappa_{\sst K}$ be the Kummer--Leopoldt constant (Definition \ref{defkappa}).

\smallskip\noindent
Then $p^{\kappa_{\sst K}}$ is the exponent of the finite group 
${\rm tor}^{}_{\Z_p} \big ({\rm log} (U_K) / {\rm log} (\ov E_K)\big )$, 
where ${\rm log}$ is the $p$-adic logarithm and
$\ov E_K$ the topological closure of $E_K$ in $U_K$ (whence the relation 
${\rm log} (\ov E_K) = \Z_p\, {\rm log} (E_K)$).
\end{theorem}

\begin{proof} Let $p^\kappa$ be the exponent of ${\rm tor}^{}_{\Z_p} 
\big ({\rm log} (U_K) / {\rm log} (\ov E_K)\big )$.

\smallskip
(i) ($\kappa$ is suitable). Let $n \gg 0$ and let $\varepsilon \in E_K$ 
be such that 
$$\hbox{$\varepsilon = u^{p^{n+\kappa}}$, $\!u \in U_K$.} $$

\noindent
So ${\rm log}(u)$ is of finite order modulo ${\rm log} (\ov E_K)$ and
${\rm log}(\varepsilon) = p^n \cdot (p^\kappa \cdot {\rm log}(u)) = p^n \cdot  {\rm log}(\ov \eta)$
with $\ov \eta \in \ov E_K$. By definition of $\ov E_K$, we can write in $U_K$, for all $N \gg 0$,
$$\hbox{$\ov \eta = \eta(N)\cdot u_N$, $\ \eta(N) \in E_K$,
$\  u_N \equiv 1\!\! \pmod {p^N}$;}$$

\noindent
we get ${\rm log}(\varepsilon) = p^n \cdot {\rm log}(\eta(N)) +p^n \cdot  {\rm log}(u_N)$ 
giving in $U_K$
$$\hbox{$\varepsilon = \eta(N)^{p^n}\! \cdot u_N^{p^n}\! \cdot \xi_N$, 
$\ \xi_N \in W_K$.} $$

\noindent
But $\xi_N$ is near $1$ (depending on the choice of $n\gg 0$), 
whence $\xi_N=1$ for all~$N$, and
$\varepsilon = \eta(N)^{p^n}\!\! \cdot u'_N$, $\ u'_N \to 1$ as $N\to\infty$;
so $u'_N = \varepsilon  \cdot \eta(N)^{-p^n}$ is a global unit, arbitrary close to $1$, 
hence, because of Leopoldt's conjecture \cite[Theorem III.3.6.2 (iii, iv)]{Gr1}, 
of the form $\varphi_N^{p^n}$ with $\varphi_N \in E_K$ 
(recall that $n$ is large enough, arbitrary, but fixed), giving 
$$\varepsilon = \eta(N)^{p^n} \cdot \varphi_N^{p^n} \in E_K^{p^n}. $$

(ii) ($\kappa$ is the least solution). Suppose there exists an integer $c < \kappa$ 
having the property given in Definition \ref{defkappa}.
Let $u_0 \in U_K$ be such that 
$$\hbox{${\rm log}(u_0)$ {\it is of order} $p^\kappa$ in
${\rm  tor}^{}_{\Z_p} \big ({\rm log} (U_K) / {\rm log} (\ov E_K)\big )$;} $$

\noindent
then ${\rm log}(u_0^{p^\kappa}) = {\rm log}(\ov \varepsilon_0)$, 
$\ov \varepsilon_0 \in \ov E_K$. This is equivalent to
$$\hbox{ $u_0^{p^\kappa} = \ov \varepsilon_0 \cdot \xi_0 = 
\varepsilon(N) \cdot u_N \cdot \xi_0,\ \ \  \varepsilon(N)\in E_K,
\ \,  u_N \equiv 1\!\! \!\pmod {p^N},  \, \ \xi_0 \in W_K$,} $$

\noindent
hence, for any $n\gg0$, $u_0^{p^{n+\kappa}} = \varepsilon(N)^{p^n} \!\!\cdot u_N^{p^n}$.
Taking $N$ large enough, but fixed, we can suppose that 
$u_N = v^{p^{2 \kappa}}$, $v \in U_K$ near $1$;
because of the above relations, ${\rm log}(v)$ is of finite order modulo 
${\rm log} (\ov E_K)$, thus  ${\rm log} (v^{p^\kappa}) \in {\rm log} (\ov E_K)$.
This is sufficient,~for 
$$u'_0 := u_0 \cdot v^{-p^{\kappa}}, $$ 

\noindent
to get ${\rm log}(u'_0)$ of 
{\it order} $p^\kappa$ modulo ${\rm log} (\ov E_K)$. So we can write:
$$\varepsilon(N)^{p^n} = u_0^{p^{n+\kappa}}\cdot u_N^{-p^n} 
= u_0^{p^{n+\kappa}}\cdot (v^{-p^\kappa})^{p^{n+\kappa}}
= u'^{p^{n+\kappa}}_0 \in U_K^{p^{n+(\kappa-c)+c }}, $$

\noindent
but, by assumption on $c$ applied to the global unit
$\varepsilon(N)^{p^n}$, we obtain 
$$\varepsilon(N)^{p^n} = \eta_0^{p^{n+(\kappa-c)}}, \ \ \eta_0 \in E_K; $$

\noindent
thus, the above relation $u'^{p^{n+\kappa}}_0 = 
\varepsilon(N)^{p^n} = \eta_0^{p^{n+(\kappa-c)}}$ yields:
$$p^c \cdot {\rm log}(u'_0) = {\rm log}(\eta_0) \in {\rm log}(E_K), $$

\noindent
which is absurd since ${\rm log}(u'_0)$ is of order $p^\kappa$ modulo
${\rm log}(\ov E_K)$.
\end{proof}

\section{Interpretation of $\kappa_{\sst K}$ --
Fundamental exact sequence}\label{section3}

The following $p$-adic result is valid without any assumption on $K$ and $p$:

\begin{lemma} \label{exact}
We have the exact sequence (from \cite[Lemma 4.2.4]{Gr1}):
$$1\to W_K \big / {\rm tor}_{\Z_p}^{}(\ov E_K) \tooo 
 {\rm tor}_{\Z_p}^{} \big(U_K \big / \ov E_K \big) 
 \mathop {\tooo}^{{\rm log}}  {\rm tor}_{\Z_p}^{}\big({\rm log}\big 
(U_K \big) \big / {\rm log} (\ov E_K) \big) \to 0. $$
\end{lemma}

\begin{proof} The surjectivity comes from the fact that if
$u \in U_K$ is such that $p^n {\rm log}(u) = {\rm log}(\ov\varepsilon)$, 
$\ov\varepsilon \in \ov E_K$, then $u^{p^n} = \ov\varepsilon \cdot \xi$ 
for $\xi \in W_K$, hence there exists $m\geq n$ such that $u^{p^m} \in \ov E_K$, 
whence $u$ gives a preimage in ${\rm tor}_{\Z_p}^{} \big(U_K \big / \ov E_K \big)$.

\smallskip
If $u \in U_K$ is such that ${\rm log}(u) \in {\rm log}(\ov E_K)$, then
$u = \ov \varepsilon \cdot \xi$ as above, giving the kernel equal to
$\ov E_K \cdot W_K /\ov E_K = W_K/ {\rm tor}_{\Z_p}^{}(\ov E_K)$.
\end{proof}

\begin{corollary}\label{leo} Let $\mu_K$ be the group of 
global roots of unity of $p$-power order of $K$. 

\noindent
Then, under the Leopoldt conjecture for $p$ in $K$, we have
${\rm tor}_{\Z_p}^{}(\ov E_K) = \mu_K$; thus in that case
$W_K \big / {\rm tor}_{\Z_p}^{}(\ov E_K) = W_K /\mu_K$.
\end{corollary}

\begin{proof}
From \cite[Corollary III.3.6.3]{Gr1}, \cite[D\'efinition 2.11, Proposition 2.12]{Ja}.
\end{proof}

Put 
$$\hbox{${\mathcal W}_K:= W_K /\mu_K\ \ $ \& $\ \ {\mathcal R}_K := 
{\rm tor}^{}_{\Z_p} \big ({\rm log} (U_K) / {\rm log} (\ov E_K)\big)$. }$$

Then the exact sequence of Lemma \ref{exact} becomes:
$$1\too {\mathcal W}_K \tooo 
\  {\rm tor}_{\Z_p}^{} \big(U_K \big / \ov E_K \big) 
 \mathop {\tooo}^{{\rm log}} {\mathcal R}_K \too 0.$$

Consider the following diagram (see \cite{Gr1}, \S\,III.2, (c), Fig.\,2.2), 
under the Leopoldt conjecture for $p$ in $K$:
\unitlength=0.84cm 
$$\vbox{\hbox{\hspace{-2.8cm} 
 \begin{picture}(11.5,5.9)
% horizontales
\put(6.5,4.50){\line(1,0){1.3}}
\put(8.75,4.50){\line(1,0){2.0}}
\put(3.85,4.50){\line(1,0){1.4}}
\put(9.1,4.15){\footnotesize$\simeq\! {\mathcal W}_K$}
\put(4.2,2.50){\line(1,0){1.25}}

\bezier{350}(3.8,4.8)(7.6,6.6)(11.0,4.8)
\put(7.2,5.8){\footnotesize${\mathcal T}_K$}

\bezier{350}(6.3,4.7)(8.6,5.5)(10.8,4.7)
\put(8.2,5.25){\footnotesize${\mathcal T}'_K$}
% vertical
\put(3.50,2.9){\line(0,1){1.25}}
\put(3.50,0.9){\line(0,1){1.25}}
\put(5.7,2.9){\line(0,1){1.25}}

\bezier{300}(3.9,0.5)(4.7,0.5)(5.6,2.3)
\put(5.2,1.3){\footnotesize$\simeq \! \Cl_K$}

\bezier{300}(6.3,2.5)(8.5,2.6)(10.8,4.3)
\put(8.4,2.7){\footnotesize$\simeq \! U_K/\ov E_K$}

\put(10.85,4.4){$H_K^{\rm pr}$}
\put(5.3,4.4){$\widetilde K\! H_K$}
\put(7.9,4.4){$H_K^{\rm reg}$}
\put(6.7,4.14){\footnotesize$\simeq\! {\mathcal R}_K$}
\put(3.3,4.4){$\widetilde K$}
\put(5.5,2.4){$H_K$}
\put(2.7,2.4){$\widetilde K \!\cap \! H_K$}%%
\put(3.4,0.40){$K$}
\end{picture}   }} $$
\unitlength=1.0cm

\noindent
where $\widetilde K$ is the compositum of the $\Z_p$-extensions, 
$\Cl_K$ the $p$-class group, $H_K$ the $p$-Hilbert class field, $H_K^{\rm pr}$ 
the maximal Abelian $p$-ramified (i.e., unramified outside~$p$) pro-$p$-extension, 
of $K$.
These definitions are given in the ordinary sense when $p=2$ (so that the real infinite 
places of $K$ are not complexified (= are unramified) in the class fields under consideration 
which are ``real'').

\smallskip
By class field theory, ${\rm Gal}(H_K^{\rm pr} / H_K) \simeq U_K/\ov E_K$ in 
which the image of ${\mathcal W}_K$ fixes $H_K^{\rm reg}$,
the Bertrandias--Payan field, ${\rm Gal}(H_K^{\rm reg} / \widetilde K)$ being 
then the Bertrandias--Payan module, except possibly if $p=2$ in 
the ``special case'' (cf. \cite{AN} about the calculation of $\kappa$ 
and the R\'ef\'erences in \cite{Gr2} for some history about this module). 

\smallskip
But ${\mathcal R}_K$ giving $\kappa_{\sst K}$ has, a priori, nothing to do with
{\it the definition} of the Bertrandias--Payan module associated with $p^r$-cyclic 
extensions of $K$, $r \geq 1$, which are embeddable in cyclic $p$-extensions of $K$ 
of arbitrary large degree.

\smallskip
Then we put
${\mathcal T}'_K := {\rm  tor}^{}_{\Z_p} ({\rm Gal}(H_K^{\rm pr} / H_K)) 
\subseteq {\mathcal T}_K :=  {\rm  tor}^{}_{\Z_p}({\rm Gal}(H_K^{\rm pr}/K))$. 
The group ${\mathcal R}_K$ is then isomorphic to 
${\rm Gal}(H_K^{\rm reg} / \widetilde KH_K)$.
Of course, for $p \geq p_0$ (explicit), ${\mathcal W}_K = [H_K : \wt K \cap H_K] = 1$,
whence ${\mathcal R}_K = {\mathcal T}_K$.
We shall see in the Section \ref{reg} that ${\mathcal R}_K \simeq 
{\mathcal T}'_K / {\mathcal W}_K$ is closely 
related to the classical $p$-adic regulator of $K$.

\begin{corollary} Under the Leopoldt conjecture for $p$ in $K$,
the Kummer--Leopoldt constant $\kappa_{\sst K}$ of $K$ is $0$ if and only if
${\mathcal R}_K=1$ (i.e., $H_K^{\rm reg} = \widetilde K H_K$).
\end{corollary}

\begin{proof} From Theorem \ref{thm} using the new terminology 
of the ``algebraic regulator'' ${\mathcal R}_K := 
{\rm tor}^{}_{\Z_p} \big ({\rm log} (U_K) / {\rm log} (\ov E_K)\big)$
whose exponent is $p^\kappa$.
\end{proof}

\begin{corollary}\label{kummer}
 If the prime number $p$ is regular, then $\kappa_{\sst K}=0$
for the field $K = \Q(\mu_p)$ of $p$th roots of unity, and any unit $\varepsilon \in E_K$
such that $\varepsilon \equiv 1 \pmod p$ is in $E_K^p$ (Kummer's lemma).
\end{corollary}

\begin{proof}
(i) We first prove that if the real unit $\varepsilon$ is congruent to $1$ 
modulo $p$ then it is a $p$th power in $U_K$.
Put $\varepsilon = 1 + \alpha \cdot p$ 
for a $p$-integer $\alpha \in K^\times$. Let $K_0$ be the maximal real subfield of $K$ 
and let $\pi_0$ be an uniformizing parameter of its $p$-completion. 
Put $\alpha = a_0+\beta\cdot \pi_0$ with $a_0 \in [0, p-1]$ and a $p$-integer $\beta$.
Since ${\rm N}_{K_0/\Q}(\varepsilon)=1$, this yields $a_0=0$, whence
$\varepsilon = 1 + \beta \cdot p \cdot \pi_0$. The valuation of $p \cdot \pi_0$,
calculated in $K$, is $p+1$, which is sufficient to get $\varepsilon \in U_K^p$
(use \cite[Proposition 5.7]{Wa1}).

\smallskip
(ii) Then we prove that $\kappa=0$. The cyclotomic field $K=\Q(\mu_p)$ 
is {\it $p$-regular and $p$-rational} in the meaning of \cite[Th\'eor\`eme 
\& D\'efinition 2.1]{GJ}, so ${\mathcal T}_K = 1$ giving $\kappa = 0$.
In other words, $\kappa = 0$ is given by a stronger condition
($p$-rationality of $K$) than ${\mathcal R}_K = 1$.

\smallskip
One may preferably use the general well-known $p$-rank formula
(the $p$-rank ${\rm  rk}^{}_{p}(A)$ of a finite Abelian group 
$A$ is the $\F_p$-dimension of $A/A^p$), valid for any field $K$ under the 
Leopoldt conjecture, when the group $\mu_K$ of $p$th roots of unity is nontrivial 
\cite[Proposition III.4.2.2]{Gr1}:
\begin{equation*}
{\rm  rk}^{}_{p}({\mathcal T}_K) = 
{\rm  rk}^{}_{p}(\Cl_K^{S_K \rm res})+\order S_K -1, 
\end{equation*}

\noindent
where $S_K$ is the set of prime ideals of $K$ above $p$
and $\Cl_K^{S_K \rm res}$ the $S_K$-class group in the restricted sense (when
$p=2$) equal to the quotient of the $p$-class group of $K$ in the restricted sense
by the subgroup generated by the classes of ideals of $S_K$; so for $K=\Q(\mu_p)$,
we immediately get ${\rm  rk}^{}_{p}({\mathcal T}_K) =  
{\rm  rk}^{}_{p}(\Cl_K)$, which is by definition trivial for regular primes.
\end{proof}

\begin{theorem} \label{nu}
Let $\kappa_{\sst K}$ be the Kummer-Leopoldt constant of $K$
(Definition \ref{defkappa}) and let $p^\nu$ be the exponent of 
${\mathcal W}_K = W_K/ \mu_K$, where
$W_K = {\rm  tor}^{}_{\Z_p}(U_K)$ and $\mu_K$ is the group of global roots of 
unity of $K$ of $p$-power order.\,\footnote{In the case $\nu=0$, if $\mu_K=1$,
then $\mu_{K_{\mathfrak p}}^{} = 1\  \forall {\mathfrak p} \in S_K$; 
if $\mu_K \ne 1$, then $S_K = \{{\mathfrak p}\}$ \& $\mu_{K_{\mathfrak p}}^{} = \mu_K^{}$.}

\smallskip
The property defining $\kappa_{\sst K}$ can be improved as follows:

\smallskip
(i)  If $\nu \geq 1$, for all $n \geq 0$, any $\varepsilon \in E_K$ such that 
$\varepsilon \in U_K^{p^{n+\kappa_{\sst K}}}$ is necessarily of the form
$\varepsilon = \zeta \cdot \eta^{p^n}$, with $\zeta \in \mu_K \cap W_K^{p^n}$,
$\eta \in E_K$.

\smallskip
(ii) If $\nu=0$, for all $n \geq 0$, any $\varepsilon \in E_K$ being in $U_K^{p^{n+\kappa_{\sst K}}}$
is necessarily in $E_K^{p^n}$.
\end{theorem}

\begin{proof} Let $n \geq 0$.
Suppose that $\varepsilon = u^{p^{n+\kappa}}$, $u \in U_K$. 
So ${\rm log}(\varepsilon) = p^n \cdot p^\kappa \cdot {\rm log}(u)=
p^n \cdot  {\rm log}(\ov \eta)$, $\ov \eta \in \ov E_K$; thus
$\ov \eta = \eta(N)\cdot u_N$, with $\eta(N) \in E_K$,
$u_N \equiv 1\!\pmod {p^N}$, for all $N \gg 0$,
and ${\rm log}(\varepsilon) = p^n \cdot {\rm log}(\eta(N)) +p^n \cdot  {\rm log}(u_N)$ 
giving in $U_K$ 
$$\hbox{$\varepsilon = \eta(N)^{p^n}\! \cdot u_N^{p^n}\! \cdot \xi_N$, 
$\ \ \xi_N \in W_K$, for all $N \gg 0$.} $$

\noindent
Taking $N$ in a suitable infinite subset of $\N$, we can suppose $\xi_N=\xi$
independent of $N \to \infty$.
Then $\xi = \big(\varepsilon \cdot \eta(N)^{-p^n}\big) \cdot u_N^{-p^n} \in {\rm tor}_{\Z_p}( \ov E_K)$, 
whence $\xi = \zeta \in \mu_K$  because of Leopoldt's conjecture (loc. cit. in proof of Corollary
\ref{leo}). Then

\medskip
\centerline{$u^{p^{n+\kappa}} = \varepsilon = \eta(N)^{p^n}\! \cdot u_N^{p^n}\! \cdot \zeta =
\eta(N)^{p^n}\! \cdot u'_N \cdot \zeta$, $\ \ u'_N (\in E_K) \to 1$ as $N \to \infty$,}

\medskip \noindent
whence $\varepsilon $ of the form $\eta(N)^{p^n}\! \cdot \varphi_N^{p^n}\! \cdot \zeta$, 
$\varphi_N \in E_K$,  for $N \gg 0$.
So $\varepsilon = \zeta \cdot \eta^{p^n}$, with $\eta \in E_K$ and
$\zeta = \varepsilon \cdot \eta^{-p^n} \in \mu_K \cap W_K^{p^n}$,
since $\varepsilon$ is a local $p^n$th power.

\smallskip
If $\nu=0$, $W_K = \mu_K$ and $\zeta \in \mu_K \cap W_K^{p^n} =\mu_K^{p^n}$
is a $p^n$th power.
\end{proof}

\section{Remarks and applications}

As above, we assume the Leopoldt conjecture for $p$ in the fields under consideration.

\smallskip\noindent
({\bf a}) The condition $\varepsilon \in U_K^{p^{n+\kappa}} = 
\plus_{{\mathfrak p} \mid p} U_{\mathfrak p}^{p^{n+\kappa}}$,
where $U_{\mathfrak p} := 1 + \ov{\mathfrak p}$, may 
be translated, in the framework of Kummer's lemma,
into a less precise condition of the form $\varepsilon \equiv 1 
\pmod {\prod_{{\mathfrak p} \mid p}{\mathfrak p}^{m_{\mathfrak p}(n, \kappa)}}$
for suitable minimal exponents $m_{\mathfrak p}(n, \kappa)$ giving 
local $p^{n+\kappa}$th powers. This was used by most of the cited 
references with $p$-adic analytical calculations using the fact that
$\order {\mathcal T}_K$ is, roughly speaking, a product
``class number'' $\times$ ``regulator'' from $p$-adic 
$L$-functions, giving an upper bound for $\kappa$
(it is the analytic way used in \cite{Wa2} and \cite{O} to generalize 
Kummer's lemma when $p$ is not regular).

\smallskip \noindent
({\bf b}) If ${\mathcal T}_K = 1$ (in which case $\kappa=0$),
the field $K$ is said to be a {\it $p$-rational field}
(see \cite[\S\,IV.3]{Gr1}, \cite{GJ}, \cite{JN}, \cite{MN}).
Then in any {\it $p$-primitively ramified} $p$-extension $L$ of $K$
(definition and examples in \cite[\S\,IV.3, (b); \S\,IV.3.5.1]{Gr1}, 
after \cite[Theorem 1, \S\,II.2]{Gr4}), 
we get ${\mathcal T}_L = 1$ whence $\kappa_{\sst L}=0$. 

\smallskip
The following examples illustrate this principle:

\smallskip
(i) The $p^m$-cyclotomic fields.
The above applies for the fields $K_m := \Q(\mu_{p^m})$ of $p^m$-roots of unity
when the prime $p$ is regular, since we have seen that ${\mathcal T}_{\Q(\mu_p)} = 1$. 

\smallskip
 (ii) Some $p$-rational $p$-extensions of $\Q$ ($p=2$ and $p=3$).
The following fields have a Kummer-Leopoldt constant $\kappa=0$ 
(\cite[Example IV.3.5.1]{Gr1}, after \cite[\S\,III]{Gr4}):

\medskip
-- The real Abelian 2-extensions of $\Q$, subfields of the fields
$\Q(\mu^{}_{2^\infty}) \cdot \Q(\sqrt{\ell}\,)$, $\ell \equiv 3 \pmod 8$, and
$\Q(\mu^{}_{2^\infty})\cdot  \Q\bigg (\hbox {$\sqrt{\sqrt{\ell}\ \Frac{a-\sqrt{\ell}}{2}}$}\,
\bigg), \  \ell = a^2+ 4\, b^2 \equiv 5 \pmod 8$.

\smallskip
-- The real Abelian $3$-extensions of $\Q$, subfields of the fields
$\Q(\mu^{}_{3^\infty}) \cdot k^{}_{\ell}$, $\ell \equiv 4,\ 7 \pmod 9$,
where $k^{}_{\ell}$ is the cyclic cubic field of conductor $\ell$.

\medskip \noindent
({\bf c}) When $\mu_K=1$, the formula giving 
${\rm  rk}^{}_{p}({\mathcal T}_K)$, used in the 
proof of Kummer's lemma (Corollary  \ref{kummer}), must be 
replaced by a formula deduced from the ``reflection theorem'': let 
$K' := K(\zeta_p)$, where $\zeta_p$ is a primitive $p$th root of unity; then
$${\rm rk}^{}_{p}({\mathcal T}_K) = {\rm rk}^{}_{\omega}
\big(\Cl_{K'}^{S_{K'}{\rm res}}\big) + \sm_{{\mathfrak p} \mid p} \delta_{\mathfrak p} - \delta , $$

\noindent
which links the $p$-rank of ${\mathcal T}_K$
to that of the $\omega$-component of the $p$-group of $S_{K'}$-ideal classes
of the field $K'$, where $\omega$ is the Teichm\"uller character of 
${\rm Gal}(K'/K)$, $\delta_{\mathfrak p} := 1$ or $0$ according as the completion 
$K_{\mathfrak p}$ contains $\zeta_p$ or not, $\delta := 1$ or $0$ according as 
$\zeta_p \in K$ or not (so that $\omega = 1$ if and only if $\zeta_p \in K$).

\smallskip \noindent
({\bf d}) Unfortunately, $p^\kappa$ may be less than $\order {\mathcal R}_K$
(hence a fortiori less than $\order {\mathcal T}_K$) due 
to the unknown {\it group structure} of ${\mathcal R}_K$; as usual, when $K/\Q$ is 
Galois with Galois group $G$, the study of its $G$-structure may give more precise 
information:

\smallskip
Indeed, to simplify assume $p>2$ unramified in $K$, so that
${\rm log}(U_K)$ is isomorphic to $O_K$, the direct sum of 
the rings of integers of the $K_{\mathfrak p}$, ${\mathfrak p} \mid p$; 
if $\eta =: 1+p \cdot \alpha \in E_K$ generates a sub-$G$-module 
of $E_K$, of index prime to $p$ (such a unit does exist since 
$E_K \otimes \Q$ is a monogenic $\Q[G]$-module; cf. 
\cite[Corollary I.3.7.2 \& Remark I.3.7.3]{Gr1}), the structure of 
${\mathcal R}_K$ can be easily deduced from the knowledge of 
$P(\alpha) \equiv \frac{1}{p}{\rm log}(\eta)$ modulo a suitable power of $p$, 
where $P(\alpha)$ is a rational polynomial expression of $\alpha$
generating a sub-$G$-module of $O_K$; 
thus many numerical examples may be obtained.

\smallskip\noindent
({\bf e}) We have given in \cite[\S\,8.6]{Gr3} a conjecture saying that, in 
any fixed number field $K$, we have ${\mathcal T}_K=1$ for all $p\gg 0$, 
giving conjecturally $\kappa = 0$ for all $p\gg 0$.

\section{Normalized $p$-adic regulator of a number field}\label{reg}

The previous Section \ref{section3} shows that the good notion of $p$-adic regulator 
comes from the expression of the $p$-adic finite group ${\mathcal R}_K$ associated 
with the class field theory interpretation of ${\rm Gal}(H_K^{\rm reg} / \widetilde KH_K)$.

\smallskip
For this, recall that $\ov E_K$ is the topological closure, in the $\Z_p$-module $U_K$
of principal local units at $p$, of the group of $p$-principal global units of $K$, and 
${\rm log}$ the $p$-adic logarithm:

\begin{definition} Let $K$ be any number field and let $p \geq 2$ be any prime number.
Under the Leopoldt conjecture for $p$ in $K$, we call 
${\mathcal R}_K := {\rm tor}^{}_{\Z_p} \big ({\rm log} (U_K) / 
{\rm log} (\ov E_K)\big)$ (or its order $\order {\mathcal R}_K$) 
the {\it normalized $p$-adic regulator} of $K$. 
\end{definition}

We have in the simplest case of {\it totally real} number fields 
(from Coates's formula \cite[Appendix]{C} and also \cite[Remarks III.2.6.5]{Gr1} for $p=2$):

\begin{proposition} \label{real}
For any totally real number field $K \ne \Q$, we have, under
the Leopoldt conjecture for $p$ in $K$,
$$\order {\mathcal R}_K \sim \frac{1}{2} \cdot
\frac{\big(\Z_p : {\rm log}({\rm N}_{K/\Q}(U_K)) \big)}
{ \order {\mathcal W}_K \cdot \prod_{{\mathfrak p} \mid p}{\rm N} {\mathfrak p}}
\cdot \frac {R_K}{\sqrt {D_K}}, $$

\noindent
where $\sim$ means equality up to a $p$-adic 
unit factor, where $R_K$ is the usual $p$-adic regulator \cite[\S\,5.5]{Wa1} and 
$D_K$ the discriminant of $K$. 
\end{proposition}

With this expression, we find again classical results obtained by means of 
analytic computations (e.g., \cite[Theorem 6.5]{A}).
In the real Galois case, with $p$ unramified in $K/\Q$, 
we get, as defined in \cite[D\'efinition 2.3]{Gr3}, 
$\ds \order {\mathcal R}_K \sim \frac{R_K}{p^{[K : \Q]-1}}$ for $p\ne 2$ and
$\ds \order {\mathcal R}_K \sim  \frac{1}{2^{d-1}} \frac{R_K}{2^{[K : \Q]-1}}$ for $p=2$, 
where $d$ is the number of prime ideals ${\mathfrak p} \mid 2$ in $K$.

\smallskip\noindent
Of course, $\order {\mathcal R}_K = \order {\rm tor}_{\Z_p}( {\rm log}(U_K))=1$ 
for $\Q$ and any imaginary quadratic field.

\end{document}